\documentclass[12pt,reqno]{amsart}
\usepackage{fullpage}
\usepackage{color}
\usepackage{amsmath}
\usepackage{amssymb}
\usepackage{bm}

\newtheorem{theorem}{Theorem}[section]
\newtheorem{lemma}[theorem]{Lemma}
\newtheorem{prop}[theorem]{Proposition}
\newtheorem{cor}[theorem]{Corollary}
\newtheorem{conjecture}[theorem]{Conjecture}
\theoremstyle{definition}

\renewcommand{\mod}[1]{{\ifmmode\text{\rm\ (mod~$#1$)}\else\discretionary{}{}{\hbox{ }}\rm(mod~$#1$)\fi}}

\newcommand{\A}{{\mathcal A}}

\newcommand{\cL}{{\mathcal L}}
\newcommand{\N}{{\mathbb N}}
\newcommand{\E}{{\mathcal E}}

\newcommand{\Z}{{\mathbb Z}}

\begin{document}

\title{Divisor sums representable as the sum of two squares}
\author[Lee Troupe]{Lee Troupe}
\address{Department of Mathematics and Computer Science, University of Lethbridge, C526 University Hall, 4401 University Drive, Lethbridge, AB, Canada T1K 3M4}
\email{lee.troupe@uleth.ca}

\begin{abstract}
Let $s(n)$ denote the sum of the proper divisors of the natural number $n$. We show that the number of $n \leq x$ such that $s(n)$ is a sum of two squares has order of magnitude $x/\sqrt{\log x}$, which agrees with the count of $n \leq x$ which are a sum of two squares. Our result confirms a special case of a conjecture of Erd{\H o}s, Granville, Pomerance and Spiro, who in a 1990 paper asserted that if $\A \subset \N$ has asymptotic density zero (e.g. if $\A$ is the set of $n \leq x$ which are a sum of two squares), then $s^{-1}(\A)$ also has asymptotic density zero.
\end{abstract}

\maketitle

\section{Introduction}

For a natural number $n$, define $s(n)$ to be the sum of the proper divisors of $n$. In a 1990 paper \cite{egps90}, Erd{\H o}s, Granville, Pomerance, and Spiro propose the following conjecture, which has motivated a substantial amount of research concerning $s(n)$.

\begin{conjecture}\label{egps conjecture}
Let $\A$ be a subset of the natural numbers of asymptotic density zero. Then $s^{-1}(\A)$ also has asymptotic density zero.
\end{conjecture}

Certain special cases of this conjecture are known to be true. If $\A$ is the set of prime numbers, Pollack shows in  \cite{pol14} that the number of $n \leq x$ such that $s(n) \in \A$ is $O(x/\log x)$. In \cite{pol15}, the same author shows that if $\A$ is the set of palindromes in any given base, then $s^{-1}(\A)$ has density zero. If one wishes to consider sets $\A$ without a prescribed structure, work of Pollack, Pomerance and Thompson  \cite{ppt18} shows that the conjecture holds for any set $\A$ whose counting function is $\ll x^{1/2 + \epsilon(x)}$, where $\epsilon(x)$ is any fixed function tending to 0 as $x \to \infty$.

Let $B(x)$ denote the count of $n \leq x$ such that $n$ can be represented as a sum of two squares. A classical result of Landau \cite{lan08} states that $B(x) \sim C x/\sqrt{\log x}$, where $C$ is an explicit constant. In particular, the set of $n \leq x$ which can be represented as the sum of two squares has asymptotic density zero, and so, in light of Conjecture \ref{egps conjecture}, it is natural to wonder whether the preimage of this set under $s(n)$ has density zero. It turns out that we can prove something stronger.

\begin{theorem}\label{main thm sum of two squares}
Let $B_s(x)$ denote the count of $n \leq x$ such that $s(n)$ can be written as a sum of two squares. Then
\[
 B_s(x) \asymp \frac{x}{(\log x)^{1/2}},
\]
where the constants implied by the $\asymp$ symbol are absolute.
\end{theorem}

In other words, not only does Conjecture \ref{egps conjecture} hold in this particular case, but in fact the order of magnitude of $B_s(x)$ is the same as that of $B(x)$.

In Section 2, we give several lemmas which lay the groundwork for the proof of Theorem \ref{main thm sum of two squares}. In Section 3, we establish a crude upper bound on the number of $n \leq x$ such that $s(n)$ is a sum of two squares, which we leverage to show that we can impose a number of additional conditions on $n$. We prove the upper bound implicit in Theorem \ref{main thm sum of two squares} in Section 4, and we prove the lower bound in Section 5.

\bigskip

\noindent\textbf{Notation.} For a positive integer $k$, we let $\log_k x$ denote the $k$th iterate of the natural logarithm. In particular, $\log_k x$ is not to be confused with the base-$k$ logarithm, and at each occurrence we assume that $x$ is large enough so that $\log_k x$ is a well-defined quantity. For a natural number $n$, we let $P(n)$ denote the largest prime factor of $n$. The function $\tau(n)$ represents the number of divisors of $n$, and $\varphi(n)$ is Euler's function, as usual. Other notation may be defined as necessary.

%

\section{Preliminaries}

Let $P(n)$ denote the largest prime factor of $n$. We define
\begin{align}\label{error set}
\E(x) := \{n \leq x : P(n) \leq x^{1/\log_2 x} \quad\text{or}\quad P(n)^2 \mid n\}.
\end{align}
Notice that each integer $n \notin \E(x)$ can be written in the form $n = mP$, where $P = P(n)$ and $P \nmid m$. We also have $s(n) = Ps(m) + \sigma(m)$, where $\sigma(m)$ denotes the sum of all divisors of $m$. It is convenient to work with such representations of $s(n)$, as they are amenable to sieve methods (see Section 3). The following lemma shows that considering only those $n \notin \E(x)$ does not discard too many values of $n$.

\begin{lemma}\label{error set size}
For sufficiently large $x$, we have $\#\E(x) \ll x/(\log x)^2$.
\end{lemma}

To prove Lemma \ref{error set size}, we require the following result of de Bruijn \cite[Theorem 2]{db66}.

\begin{prop}\label{smooth ub}
Let $x \geq y \geq 2$ satisfy $(\log x)^2 \leq y \leq x$. Whenever $u := \frac{\log x}{\log y} \to \infty$, we have
\[
\Psi(x, y) \leq x/u^{u + o(u)}.
\]
\end{prop}

\begin{proof}[Proof of Lemma \ref{error set size}.]
If $n \in \E(x)$, then either $P(n) \leq x^{1/\log_2 x}$ or $P(n) > x^{1/\log_2 x}$ and $P(n)^2 \mid n$. By Proposition \ref{smooth ub}, the number of $n \leq x$ for which the former possibility holds is $\ll x/(\log x)^2$. The number of $n \leq x$ for which the latter possibility holds is
\[
\ll x \sum_{p > x^{1/\log_2 x}} \frac{1}{p^2} \ll xe^{-\log x/\log_2 x},
\]
and so the count of these $n$ is also $\ll x/(\log x)^2$.
\end{proof}

In the proof of Theorem \ref{main thm sum of two squares}, we make frequent use of a generalization of Mertens's theorem to arithmetic progressions. The version stated below follows from a result proved by Mertens himself (cf. \cite[pp. 41--43, 449--450]{lan53}). Note that although there is a dependence on the modulus $b$ and the residue class $a \mod b$ in the constants in the following two results, we will work exclusively in the case $b = 4$ and $a = 1$ or $a = 3$.

\begin{theorem}\label{additive mertens in progressions}
For fixed $a, b \in \Z$,
\[
\sum_{\substack{p \leq x \\ p \equiv a \mod b}} \frac{1}{p} = \frac{1}{\varphi(b)} \log\log x + c_{a, b} + O_b\bigg(\frac{1}{\log x}\bigg),
\]
where subscripts indicate dependence upon $a$ and $b$.
\end{theorem}

Using the fact that $\log(1 - 1/p) = -\frac{1}{p} + O(1/p^2)$, one quickly deduces the following corollary. Though stronger versions of this theorem exist in the literature (e.g. \cite{wil74}), this weak result is sufficient for us.

\begin{cor}\label{mult mertens in progressions}
For fixed $a, b \in \Z$,
\[
\prod_{\substack{p \leq x \\ p \equiv a \mod b}} \bigg(1 - \frac{1}{p} \bigg) \asymp \frac{1}{(\log x)^{1/\varphi(b)}},
\]
where the constants implied by the $\asymp$ symbol depend only on $a$ and $b$.
\end{cor}

Let $n$ be a natural number. Suppose $s(n)$ is a sum of two squares. Then by a classical theorem of Fermat, if $R$ is the largest factor of $s(n)$ supported on the primes $p \equiv 3 \mod 4$, then $R$ is a square. The following lemma plays a crucial role in the proof of the upper bound implicit in Theorem \ref{main thm sum of two squares}. It provides an upper bound on the number of $n \leq x$ whose prime factors belong to a certain interval $[z, y]$, with the possible exception of primes $p \leq z$ dividing this square factor $R$.

\begin{lemma}\label{smooth lemma 3 mod 4}
Let $x \geq y \geq z \geq 2$, and let $R$ be a positive integer. Let $\Psi(x)$ denote the count of numbers $n \leq x$ such that $n$ is $y$-smooth and such that, if $p \mid n$ with $p \leq z$ and $p \equiv 3 \mod 4$, then $p \mid R$. Then
\[
\Psi(x) \ll \frac{x}{\sqrt{\log z}} e^{-u/2} \frac{R}{\varphi(R)}, \quad \text{where} \quad u := \frac{\log x}{\log y}.
\]
\end{lemma}

A key ingredient in the proof of Lemma \ref{smooth lemma 3 mod 4} is a result of Halberstam and Richert \cite{hr74}, which we state in precisely the same form as \cite[Lemma 2.4]{pol14} (which itself is derived from \cite[Corollary 5.1, p. 309]{ten95}).

\begin{lemma}\label{smooth lemma ingredient}
Let $f$ be a real-valued, nonnegative multiplicative function. Suppose there are positive constants $\lambda_1$ and $\lambda_2$, with $\lambda_2 < 2$, so that $f(p_k) \leq \lambda_1\lambda_2^{k-1}$ for all prime powers $p^k$. Then for all $x \geq 1$,
\[
\sum_{n \leq x} f(n) \ll_{\lambda_1, \lambda_2} x \exp \bigg( \sum_{p \leq x} \frac{f(p) - 1}{p} \bigg).
\]
\end{lemma}

\begin{proof}[Proof of Lemma \ref{smooth lemma 3 mod 4}.]
Let $\chi(n)$ denote the characteristic function of those $n$ counted by $\Psi$. Then for any $\alpha > 0$,
\begin{align*}
\Psi(x) = \sum_{n \leq x} \chi(n) &\leq x^{3/4} + \sum_{x^{3/4} < n \leq x} \chi(n) \\
&\leq x^{3/4} + x^{-3\alpha/4} \sum_{n \leq x} \chi(n)n^\alpha.
\end{align*}
Let $\alpha = 2/(3\log y)$. We estimate this final sum by applying Lemma \ref{smooth lemma ingredient} with $f(n) = \chi(n)n^\alpha$. Notice that $f(p^k) = 0$ if $p > y$ or if $p < z$ with $p \equiv 3 \mod 4$ and $p \nmid R$, and $f(p^k) = p^{k\alpha}$ if $p \in [z, y]$ or if $p < z$ with $p \equiv 3 \mod 4$ and $p \mid R$. Since $p^{k\alpha} \leq \exp(2k/3)$, the hypotheses of Lemma \ref{smooth lemma ingredient} are satisfied with $\lambda_1 = \lambda_2 = \exp(2/3) < 2$. Thus,
\begin{align}\label{psi lemma eqn 1}
\sum_{n \leq x} \chi(n)n^\alpha &\ll x\exp\bigg( -\sum_{\substack{p \leq z \\ p \equiv 3 \mod 4 \\ p \nmid R}} \frac{1}{p} \bigg) \exp\bigg( \sum_{p \leq y} \frac{p^\alpha - 1}{p} \bigg) \nonumber \\
&\ll \frac{x}{\sqrt{\log z}} \exp\bigg( \sum_{p \mid R} \frac{1}{p} \bigg) \exp\bigg( \sum_{p \leq y} \frac{p^\alpha - 1}{p} \bigg)
\end{align}
using Theorem \ref{additive mertens in progressions}. Now,
\[
\exp\bigg( \sum_{p \mid R} \frac{1}{p} \bigg) \leq \exp\bigg(- \sum_{p \mid R} \log\bigg(1 - \frac{1}{p}\bigg)\bigg) = \prod_{p \mid R} \bigg(1 - \frac{1}{p}\bigg)^{-1} = \frac{R}{\varphi(R)}.
\]
For the second exponential factor in \eqref{psi lemma eqn 1}, we have that $\alpha\log p \ll 1$ for $p \leq y$, and so $p^\alpha - 1 \ll \alpha \log p$. Therefore, $\sum_{p \leq y} \frac{p^\alpha - 1}{p} \ll \alpha \sum_{p \leq y} \frac{\log p}{p} \ll \alpha \log y \ll 1$. Putting it all together, we have shown
\begin{align}\label{psi almost finished}
\Psi(x) \ll x^{3/4} + x^{-3\alpha/4}\frac{x}{\sqrt{\log z}} \frac{R}{\varphi(R)}.
\end{align}
With $u$ and $\alpha$ as previously defined, we have $x^{-3\alpha/4} = e^{-u/2}$. Also, if $y \geq 11$, then since $x \geq z$,
\[
\frac{x}{\sqrt{\log z}} e^{-u/2} \geq \frac{x}{\log x} x^{-1/(2\log 11)} \gg x^{0.79}.
\]
(If $y \leq 11$, then $\Psi(x) = O((\log x)^4)$.) Thus the second term in \eqref{psi almost finished} dominates, and the proof is complete.
\end{proof}

Next, we present a lemma which facilitates the proof of the upper bound implicit in Theorem \ref{main thm sum of two squares} by allowing us to assume that the square factor $R$ of $s(n)$ is not too big.

\begin{lemma}\label{square factor size}
Let $n \leq x$ with $n \notin \E(x)$, so that $n = mP$ with $P > x^{1/\log_2x}$ and $P \nmid m$. Let $R$ be the largest divisor of $s(n)$ supported on the primes $p \equiv 3 \mod 4$ not exceeding $(x/m)^{1/10}$. The number of such $n \leq x$ with $R > (\log x)^2$ is $o(x/\sqrt{\log x})$.
\end{lemma}

\begin{proof}
Suppose $n = mP$ is as above with $R > (\log x)^2$. Then $R$ is a square. Since every prime factor of $R$ is at most $(x/m)^{1/10}$, if $R > (x/m)^{1/4}$, then by removing factors of the form $p^2$ from $R$ one at a time, we will eventually obtain a square factor of $s(n)$ that is at most $(x/m)^{1/4}$ but greater than $(x/m)^{1/4 - 2/10} = (x/m)^{1/20} > (\log x)^2$. So, if $R > (\log x)^2$, there is a square factor $S$ of $s(n)$ with $(\log x)^2 < S \leq (x/m)^{1/4}$.

Write $S = S_1 S_2$, where $S_1 = (S, s(m))$ and $S_2 = S/S_1$. We count the number of $n \leq x$ corresponding to a fixed factorization $S = S_1S_2$, then sum on possible factorizations. Since $S \mid s(n) = Ps(m) + \sigma(m)$, $S_1 \mid \sigma(m)$, and so $S_1 \mid \sigma(m) - s(m) = m$. Also, $S_2 = S/(S, s(m))$ is coprime to $s(m)/(S, s(m)) = s(m)/S_1$, and so the congruence
\[
P\frac{s(m)}{S_1} \equiv -\frac{\sigma(m)}{S_1} \mod{S_2}
\]
places $P$ in a uniquely determined residue class modulo $S_2$. Note also that $P \leq x/m$, which means that $x/m > x^{1/\log_2 x}$. By the Brun--Titchmarsh inequality \cite[Theorem 3.8, p. 110]{hr74}, the number of choices for $P$ given $m$ is
\[
\ll \frac{x/m}{\varphi(S_2)\log(\frac{x}{mS_2})} \ll \frac{x \log_2 x}{m \varphi(S_2) \log x},
\]
since $\frac{x}{mS_2} \geq \frac{x}{mS} \geq (\frac{x}{m})^{3/4} \geq x^{3/4\log_2 x}$ (here we use the fact that $S \leq (x/m)^{1/4}$). Summing on $m \leq x^{1 - 1/\log_2x}$ that are multiples of $S_1$ shows that the number of possible values of $n = mP$ corresponding to the decomposition $S = S_1S_2$ is
\[
\ll \frac{x \log_2 x}{S_1 \varphi(S_2)}.
\]
We sum now on decompositions $S = S_1 S_2$. Using the fact that $S_2/\varphi(S_2) \leq S/\varphi(S)$,
\[
\sum_{S_1S_2 = S} \frac{1}{S_1 \varphi(S_2)} = \frac{1}{q} \sum_{S_2 \mid S} \frac{S_2}{\varphi(S_2)} \leq \frac{1}{S} \bigg( \tau(S) \frac{S}{\varphi(S)} \bigg) = \frac{\tau(S)}{\varphi(S)},
\]
for a total count of $n \leq x$ that is
\[
\ll x \log_2 x \frac{\tau(S)}{\varphi(S)}.
\]
Summing this quantity over squares $S > (\log x)^2$ gives a final upper bound of
\[
\ll x/(\log x)^{1 + o(1)} = o(x/\sqrt{\log x}).
\]
\end{proof}

We conclude Section 2 with three lemmas, which are Lemmas 2.1, 2.2 and 2.7, respectively, of \cite{pol14}. We refer the reader to that paper for their proofs.

\begin{lemma}\label{paul's 2.1}
Let $x \geq 3$, and let $q$ be a positive integer. The number of $n \leq x$ for which $q \nmid \sigma(n)$ is
\[
\ll x/(\log x)^{1/\varphi(q)}
\]
uniformly in $q$.
\end{lemma}

\begin{lemma}\label{paul's 2.2}
For each prime $p$, the number of $n \leq x$ for which $p \mid \sigma(n)$ is
\[
\ll \frac{x \log_2 x}{p^{1/2}}.
\]
\end{lemma}

\begin{lemma}\label{paul's 2.7}
Let $q$ be a natural number with $q \leq x^{\frac{1}{2\log_3 x}}$. The number of $n \leq x$ not belonging to $\E(x)$ for which $q \mid s(n)$ is
\[
\ll \frac{\tau(q)}{\varphi(q)} \cdot x\log_3 x.
\]
\end{lemma}

\section{A crude, but useful, upper bound}

For $n \leq x$ with $n \not\in \E(x)$, we have $n = mP$ where $P > x^{1/\log_2x}$ and $P \nmid m$. The purpose of the present section is to establish the following proposition, which allows us to impose a number of additional conditions on $m$ without sacrificing too many values of $n$. Our proof borrows many ideas from an argument of Pollack \cite[Section 5.2]{pol14}. 

\begin{prop}\label{m conditions}
For each $n \leq x$, write $n = mP$, where $P := P(n)$. Consider those $n \leq x$ such that $s(n)$ can be written as a sum of two squares. For all but $o(x/\sqrt{\log x})$ such values of $n$, all of the following conditions hold:
\begin{itemize}

\item[(i)] $\displaystyle \prod_{p \leq \sqrt{\log_2 x}} p$ divides $\sigma(m)$

\item[(ii)] $\displaystyle \sum_{\substack{p \mid \sigma(m) \\ p > (\log_2 x)^{10}}} \frac{1}{p} \leq 1$

\item[(iii)] $\displaystyle \sum_{\substack{p \mid s(m) \\ p > (\log_2 x)^{10}}} \frac{1}{p} \leq 1$.

\end{itemize}

\end{prop}

\begin{proof}

First, let us assume that $n \notin \E(x)$; by Lemma \ref{error set size}, this discards $O(x/(\log x)^2)$ values of $n$. Suppose $s(n)$ can be written as the sum of two squares. If $Q$ is the largest divisor of $s(n)$ supported on the primes $p \equiv 3 \mod 4$, then $Q$ must be a square. Write $n = mP$, with $P = P(n)$ and $P \nmid m$; then $s(n) = Ps(m) + \sigma(m)$.

We begin by bounding the number of $n$ corresponding to a given $m$. By an application of Selberg's upper bound method (specifically \cite[Theorem 4.2]{hr74}), the number of possibilities for the prime $P \leq x/m$ is
\begin{align}\label{selberg ub sieve}
\ll \frac{x/m}{\log \tfrac{x}{m}} \prod_{\substack{p \leq x/m \\ p \equiv 3 \mod 4 \\ p \nmid Q}} \bigg(1 - \frac{\rho(p)}{p} \bigg) \prod_{\substack{p \leq x/m \\ p \equiv 3 \mod 4 \\ p \nmid Q \\ p \mid \sigma(m) }} \bigg(1 - \frac{1}{p} \bigg)^{-1},
\end{align}
where
\[
\rho(p) = \#\{n \mod p : ns(m) + \sigma(m) \equiv 0 \mod p\}.
\]

We observe that $(m, \sigma(m))$ is not divisible by any primes $p \equiv 3 \mod 4$, $p \nmid Q$. Indeed, $s(n)/Q$ is free of such prime factors. But
\begin{align}\label{free of primes}
(m, \sigma(m)) \mid (P+1)\sigma(m) - mP = s(n),
\end{align}
so that $s(n)$ cannot be free of prime factors $p \equiv 3 \mod 4$, $p \nmid Q$ unless the same is true for $(m, \sigma(m))$. Now, for a prime $p \equiv 3 \mod 4$ and $p \nmid Q$, if $p \mid s(m)$ and $p \mid \sigma(m)$, then $p \mid m$, which contradicts $(m, \sigma(m))$ being free of such prime factors. Thus, for such primes $p$, $\rho(p) = 0$ if $p \mid s(m)$ and $\rho(p) = 1$ otherwise. Therefore, the expression \eqref{selberg ub sieve} is
\begin{align}\label{selberg ub sieve 2}
\ll \frac{x/m}{\log \tfrac{x}{m}}\prod_{\substack{p \leq x/m \\ p \equiv 3 \mod 4 \\ p \nmid Q}} \bigg(1 - \frac{1}{p} \bigg) &\prod_{\substack{p \mid s(m)\sigma(m) \\ p \equiv 3 \mod 4 \\ p \nmid Q}} \bigg(1 + \frac{1}{p}\bigg) \nonumber \\  &\ll \frac{x/m}{\log \tfrac{x}{m}}\prod_{\substack{p \leq x/m \\ p \equiv 3 \mod 4}} \bigg(1 - \frac{1}{p} \bigg) \prod_{\substack{p \mid Qs(m)\sigma(m) \\ p \equiv 3 \mod 4}} \bigg(1 + \frac{1}{p}\bigg),
\end{align}
where we have removed the condition $p \nmid Q$ from the first product at the cost of allowing $p \mid Q$ in the second product. By Corollary \ref{mult mertens in progressions}, the first product is $\ll 1/\sqrt{\log x/m}$, and so \eqref{selberg ub sieve 2} becomes

\begin{align}\label{selberg ub sieve 3}
\ll \frac{x/m}{(\log \tfrac{x}{m})^{3/2}} \prod_{\substack{p \mid Qs(m)\sigma(m) \\ p \equiv 3 \mod 4}} \bigg(1 + \frac{1}{p}\bigg).
\end{align}

Note that $\prod_{p \mid Qs(m)\sigma(m)} (1 + 1/p) \leq \frac{\sigma(Qs(m)\sigma(m))}{Qs(m)\sigma(m)}$. Since $Qs(m)\sigma(m) \leq s(n)^3 < x^3$, we have $\frac{\sigma(Qs(m)\sigma(m))}{Qs(m)\sigma(m)} \ll \log_2(x)$ by maximal order results for the $\sigma$-function. Also, since $P > x^{1/\log_2 x}$, we have that $m < x/\exp(\log x / \log_2 x)$. Therefore $\log(x/m) > \log x / \log_2 x$. Using these facts together with Corollary \ref{mult mertens in progressions}, the number of possibilities for the prime $P \leq x/m$ satisfies
\begin{align}\label{crude bound}
\ll \frac{x(\log_2 x)^{5/2}}{m(\log x)^{3/2}}.
\end{align}

Using this upper bound, Proposition \ref{m conditions} follows from the same arguments appearing in the proof of Theorem 1.11 in \cite{pol14}; for completeness, we repeat these arguments here. First, we may assume
\begin{align}\label{m isnt small}
m \geq \exp(\log x/ (\log_2 x)^4),
\end{align}
since summing \eqref{crude bound} over those $m$ less than this bound gives a count of corresponding $n$ that is $o(x/\sqrt{\log x})$. Let $L = \lceil \log x \rceil$, and choose $\lambda \in \cL = \{1/L, \ldots, (L - 1)/L\}$ as large as possible so that $m > x^{\lambda}$; a simple calculation shows that $m \in (x^\lambda, ex^\lambda]$. By Lemma \ref{paul's 2.1}, the number of $m \in (x^\lambda, ex^\lambda]$ for which (i) fails is
\begin{align*}
&\ll \sum_{d \leq \sqrt{\log_2 x}} \frac{x^\lambda}{(\log x^\lambda)^{1/\varphi(d)}} \ll x^\lambda \sum_{d \leq \sqrt{\log_2 x}} \exp(-\log_2(x^\lambda) / \sqrt{\log_2 x} ) \\
&\ll x^\lambda / \exp(\tfrac12 \sqrt{\log_2 x} ) \ll x^\lambda/(\log_2 x)^4.
\end{align*}
Here, we use that $m \asymp x^\lambda$ and that $m$ satisfies \eqref{m isnt small}. From Lemma \ref{paul's 2.2}, the count of $m \in (x^\lambda, ex^\lambda]$ for which (ii) fails is
\begin{align*}
\ll \sum_{m \leq ex^\lambda} \sum_{\substack{p \mid \sigma(m) \\ p > (\log_2 x)^{10}}} \frac{1}{p} &\leq \sum_{p > (\log_2 x)^{10}} \sum_{\substack{m \leq ex^\lambda \\ p \mid \sigma(m)}} 1 \\
&\ll x^{\lambda}\log_2 x \sum_{p > (\log_2 x)^{10}} \frac{1}{p^{3/2}} \ll \frac{x^\lambda}{(\log_2 x)^4}.
\end{align*}

For (iii), since $s(m) < \sigma(m) < x$, 
\[
\sum_{\substack{p \mid s(m) \\ p > \log x}} \frac{1}{p} \leq \frac{1}{\log x} \sum_{\substack{p \mid s(m) \\ p > \log x}} 1 \leq \frac{1}{\log x} \cdot \frac{\log s(m)}{\log_2 x} \leq \frac{1}{\log_2 x}.
\]
Hence, if $m$ fails (iii), then for large enough $x$
\begin{align}\label{failing (iii)}
\sum_{\substack{p \mid s(m) \\ (\log_2 x)^{10} < p \leq \log x}} \frac{1}{p} \geq \frac{1}{2}.
\end{align}
The size of the set $\E(x^\lambda)$ is $O(x^\lambda/(\log_2 x)^4)$, while the number of $m \in (x^\lambda, ex^\lambda]$ not belonging to $\E(x^\lambda)$ and satisfying \eqref{failing (iii)} is at most
\begin{align*} 
2 \sum_{\substack{m \leq ex^\lambda \\ m \notin \E(x^\lambda)}} \sum_{\substack{p \mid s(m) \\ (\log_2 x)^{10} < p \leq \log x}} \frac{1}{p} &= 2 \sum_{(\log_2 x)^{10} < p \leq \log x} \frac{1}{p} \sum_{\substack{m \leq ex^\lambda \\ m \notin \E(x^\lambda) \\ p \mid s(m)}} 1 \\
&\ll x^{\lambda}\log_3 x \sum_{p > (\log_2 x)^{10}} \frac{1}{p^2} \ll \frac{x^\lambda}{(\log_2 x)^9}.
\end{align*}
Here we applied Lemma \ref{paul's 2.7} with $x$ replaced by $ex^\lambda$ and $q$ replaced by $p$. Collecting our estimates, we see that the number of $m \in (x^{\lambda}, ex^{\lambda}]$ failing one of (i), (ii), or (iii) is
\begin{align}\label{m in lambda failing}
O\bigg( \frac{x^{\lambda}}{(\log_2 x)^4} \bigg).
\end{align}
Summing the bound \eqref{crude bound} over these $m$, we see that the number of corresponding $n$ is
\[
\frac{x(\log_2 x)^{5/2}}{(\log x)^{3/2}} \cdot \frac{1}{x^{\lambda}} \cdot \frac{x^{\lambda}}{(\log_2 x)^4} \ll \frac{x}{(\log x)^{3/2} (\log_2 x)^{5/2}}.
\]
Summing over the $O(\log x)$ values of $\lambda$ completes the proof.
\end{proof}

\section{Proof of upper bound in Theorem \ref{main thm sum of two squares}}

In this section, we establish an upper bound for the number of $n \leq x$, $n = mP$, where $P > x^{1/\log_2 x}$ and $m$ satisfies conditions (i) -- (iii) from Proposition \ref{m conditions}, such that $s(n)$ is the sum of two squares. Let $\lambda \in \{0, 1/L, \ldots, (L - 1)/L\}$, where $L = \lceil \log x \rceil$. Then $\lambda < \log m / \log x \leq \lambda + 1/L$, so that $x^\lambda < m \leq ex^\lambda$.

Let $R$ be the largest divisor of $s(n)$ supported on the primes $p \equiv 3 \mod 4$ not exceeding $(x/m)^{1/10}$. Note that this means $R$ is a square. By Lemma \ref{square factor size}, we can assume that $R \leq (\log x)^2$, as this discards a negligible number of $n$. We will count those $n \leq x$ corresponding to a given $R$ and, at the end, sum on square numbers $R$ with the aforementioned property.

If $R \mid s(n)$ and $s(n) = Ps(m) + \sigma(m)$, then $\gcd(R, s(m)) \mid \sigma(m)$. Fix a decomposition $R = R_1 R_2$, where $R_1 = \gcd(R, s(m))$ and $R_2 = R/R_1$; then
\[
R_2 \mid P \frac{s(m)}{R_1} + \frac{\sigma(m)}{R_1}.
\]
Therefore
\[
P \frac{s(m)}{R_1} \equiv -\frac{\sigma(m)}{R_1} \mod{R_2}.
\]
Since $\frac{s(m)}{R_1}$ and $R_2$ are coprime, $P$ belongs to a uniquely determined residue class $a \mod{R_2}$, where $0 \leq a < R_2$ and $a$ depends on $m$ and $R$. So we count the number of $P \equiv a \mod{R_2}$ for which $Ps(m) + \sigma(m)$ is free of prime factors $q \equiv 3 \mod 4$, $q \leq (x/m)^{1/10}$, $q$ not dividing $R$.

Note that $P = a + R_2u$, where $u \leq \frac{x}{mR_2}$. We now claim that for each prime $q \leq (x/m)^{1/10}$ with $q \nmid R$ and $q \nmid s(m)\sigma(m)$, the number $u$ avoids two distinct residue classes modulo $q$. Indeed, if $P > (x/m)^{1/10}$, then $P = a + R_2 u$ is not divisible by any primes $\leq (x/m)^{1/10}$. Thus, for each prime $q \leq (x/m)^{1/10}$ with $q \nmid R$, we have that $u$ avoids the residue class $-R_2^{-1}a \mod q$. In addition, for $n$ to correspond to $R$, we need that if $q \leq (x/m)^{1/10}$, with $q \nmid R$ and $q \equiv 3 \mod 4$,
\[
(a + R_2 u)s(m) + \sigma(m)
\]
is not divisible by $q$. This means that $u$ avoids the residue class $R_2^{-1}(\sigma(m)s(m)^{-1} - a) \mod q$ -- which is distinct from the residue class $-R_2^{-1}a \mod q$ that $u$ already avoids -- provided $q \nmid s(m)\sigma(m)$.

Since $R_2 \leq R \leq (\log x)^2$, the number of such $P$ is, by Brun's sieve,
\[
\ll \frac{x}{m R_2 \log(x/m)^{3/2}} \bigg( \frac{R}{\varphi(R)} \bigg)^2 \prod_{\substack{q \mid s(m) \sigma(m) \\ q \equiv 3 \mod 4 \\ q \nmid R}} \bigg(1 + \frac{1}{q} \bigg).
\]
Now, we work with the product. We can write this product as
\[
\prod_{\substack{q \mid s(m)\sigma(m) \\ q \equiv 3 \mod 4 \\ q \nmid R}} \bigg(1 + \frac{1}{q}\bigg) = P_1 P_2 P_3,
\]
where
\[
P_1 = \prod_{\substack{q \mid s(m)\sigma(m) \\ q \equiv 3 \mod 4 \\ q \leq \sqrt{\log_2 x} \\ q \nmid R}} \bigg(1 + \frac{1}{q}\bigg), \qquad P_2 = \prod_{\substack{q \mid s(m)\sigma(m) \\ q \equiv 3 \mod 4 \\ q > (\log_2 x)^{10}  \\ q \nmid R}} \bigg(1 + \frac{1}{q}\bigg), \qquad P_3 = \prod_{\substack{q \mid s(m)\sigma(m) \\ q \equiv 3 \mod 4 \\ \sqrt{\log_2 x} < q \leq (\log_2 x)^{10} \\ q \nmid R}} \bigg(1 + \frac{1}{q}\bigg).
\]

By taking logarithms and using Mertens' theorem, conditions (ii) and (iii) give $P_2 \ll 1$, and the same technique shows $P_3 \ll 1$ as well. It remains to estimate $P_1$. We have
\[
P_1 = \prod_{\substack{q \mid s(m)\sigma(m) \\ q \equiv 3 \mod 4 \\ q \leq \sqrt{\log_2 x} \\ q \nmid R}} \bigg(1 + \frac{1}{q}\bigg) \leq \prod_{\substack{q \mid s(m) \\ q \equiv 3 \mod 4 \\ q \leq \sqrt{\log_2 x} \\ q \nmid R}} \bigg(1 + \frac{1}{q}\bigg) \prod_{\substack{q \mid \sigma(m) \\ q \equiv 3 \mod 4 \\ q \leq \sqrt{\log_2 x} \\ q \nmid R}} \bigg(1 + \frac{1}{q}\bigg).
\]
By condition (i) above, the second product on the right-hand side runs over all primes $q \leq \sqrt{\log_2 x}$. Upon taking logarithms and using Mertens's theorem for progressions, one sees that this product is $\ll \sqrt{\log_3 x}$. As for the first product on the right-hand side, recall that $(m, \sigma(m)) = (s(m), \sigma(m))$ is free of prime factors $q \equiv 3 \mod 4$; hence, this first product is empty, by (i). Therefore
\[
\prod_{\substack{q \mid s(m) \sigma(m) \\ q \equiv 3 \mod 4 \\ q \nmid R}} \bigg(1 + \frac{1}{q} \bigg) \ll \sqrt{\log_3 x},
\]
and so the number of such $P$ is
\begin{align}\label{upper bound count of P}
\ll \frac{x^{1 - \lambda} \sqrt{\log_3 x}}{(1 - \lambda)^{3/2} (\log x)^{3/2}} \bigg( \frac{R}{\varphi(R)} \bigg)^2 \frac{1}{R_2}.
\end{align}

Now we count $m$ corresponding to $\lambda$. Since $R_1 \mid s(m)$ and $R_1 \mid \sigma(m)$, we have $R_1 \mid m$. Writing $m = R_1 m'$, it suffices to count $m'$. We have $m' \leq ex^{\lambda}/R_1$, $m'$ is $x^{1 - \lambda}$-smooth, and $m'$ has no prime factors $q \equiv 3 \mod 4$ with $q \leq \sqrt{\log_2 x}$, except possibly those such $q$ with $q \mid R$. By Lemma \ref{smooth lemma 3 mod 4}, the number of such $m'$ is
\begin{align}\label{upper bound count of m}
\ll \frac{x^\lambda}{R_1 \sqrt{\log_3 x}}(1 - \lambda)^2 \frac{R}{\varphi(R)}.
\end{align}

By multiplying \eqref{upper bound count of P} and \eqref{upper bound count of m}, we see that the overall count of $n$ is
\[
\ll \frac{x (1 - \lambda)^{1/2}}{R_1 R_2 (\log x)^{3/2}} \bigg( \frac{R}{\varphi(R)} \bigg)^3.
\]
Summing on decompositions $R = R_1 R_2$ and then on square numbers $R < (\log x)^2$ gives
\begin{align*}
\sum_{\substack{R < (\log x)^2 \\ R = \square}} \sum_{R_1R_2 = R} \frac{x (1 - \lambda)^{1/2}}{R_1 R_2 (\log x)^{3/2}} \bigg( \frac{R}{\varphi(R)} \bigg)^3 &= \sum_{\substack{R < (\log x)^2 \\ R = \square}} \frac{x (1 - \lambda)^{1/2}}{(\log x)^{3/2}} \bigg( \frac{R}{\varphi(R)} \bigg)^3 \frac{\tau(R)}{R} \\ &\ll \frac{x}{(\log x)^{3/2}}.
\end{align*}
Finally, summing on the $O(\log x)$ values of $\lambda$ yields the desired bound. \qed

\section{Proof of lower bound in Theorem \ref{main thm sum of two squares}}

The aim of this section is to establish a lower bound of the form
\[
\sum_{\substack{n \leq x \\ s(n) = \square + \square}} 1 \gg \frac{x}{\sqrt{\log x}}.
\]


We achieve this lower bound by imposing a number of conditions on $n \leq x$. We assume that $n \notin \E(x)$, so that $n = mP$ where $P := P(n) > x^{1/\log_2 x}$ and $P \nmid m$. We also require that $m$ satisfies the bound \eqref{m isnt small}, condition (i) of Proposition \ref{m conditions}, and that if a prime $p$ divides $(m, \sigma(m))$, then $p \equiv 1 \mod 4$. Further, we count only those $m$ that can be decomposed as $m = m_1m_2$, where $m_1$ and $m_2$ satisfy:

\begin{itemize}

\item $m_1$ squarefree 

\item $m_1$ is $\sqrt{\log_2 x}$-smooth

\item $p \mid m_1 \implies p \equiv 1 \mod 4$

\item $m_2 \equiv 3 \mod 4$

\item $p \mid m_2 \implies p >(\log_2 x)^{10}$.

\end{itemize}


With this setup, we bound from below the sum
\[
\sideset{}{'} \sum_{m \leq x^{1/500}} \sum_{\substack{x^{1/2} < P \leq x/m \\ p \mid Ps(m) + \sigma(m) \implies p \equiv 1 \mod 4}} 1,
\]
where the $'$ on the sum indicates that $m$ satisfies all conditions mentioned above. Note that $(m, \sigma(m)) = (s(m), \sigma(m))$ can be written as a sum of two squares, as $(m, \sigma(m))$ is free of prime factors $q \equiv 3 \mod 4$. Since the set of numbers which are a sum of two squares is closed under multiplication, it suffices to estimate
\[
\sideset{}{'} \sum_{m \leq x^{1/500}} \sum_{\substack{x^{1/2} < P \leq x/m \\ p \mid P\alpha(m) + \beta(m) \implies p \equiv 1 \mod 4}} 1
\]
from below, where $\alpha(m) = s(m)/(m, \sigma(m))$ and $\beta(m) = \sigma(m)/(m, \sigma(m))$.

We handle the inner sum by appealing to a theorem tailor-made for our problem, namely \cite[Theorem 14.8]{fi10}, stated below.

\begin{theorem}\label{operadecribro}
Let $\alpha \geq 1, \beta \neq 0, (\alpha, \beta) = 1$, and $\alpha + \beta \equiv 1 \mod 4$. Then, for $x \geq 2|\beta|\alpha^{-1}$, $x \geq \alpha^{68}$, we have
\[
\sum_{\substack{p \leq x \\ q \mid \alpha p + \beta \implies q \equiv 1 \mod 4 }} 1 \asymp \frac{x}{(\log x)^{3/2}} \prod_{\substack{p \mid \alpha\beta \\ p \equiv 3 \mod 4}} \bigg( 1 + \frac{1}{p} \bigg).
\]
\end{theorem}

Since $m = m_1m_2 \equiv 3 \mod 4$, we have that $\alpha(m) + \beta(m) \equiv 1 \mod 4$. We also have $(\alpha(m), \beta(m)) = 1$, and $\alpha(m) \leq s(m) \leq m^2 \leq x^{1/250}$. Hence, all of the hypotheses of Theorem 14.8 are satisfied, and so
\[
\sideset{}{'} \sum_{m \leq x^{1/500}} \sum_{\substack{x^{1/2} < P \leq x/m \\ p \mid P\alpha(m) + \beta(m) \implies p \equiv 1 \mod 4}} 1 \gg \sideset{}{'} \sum_{m \leq x^{1/500}} \frac{x/m}{(\log x)^{3/2}} \prod_{\substack{p \mid \alpha(m)\beta(m) \\ p \equiv 3 \mod 4}} \bigg(1 + \frac{1}{p} \bigg).
\]
Consider the product in the above display. Noting that $(\alpha(m), \beta(m)) = 1$ and $(s(m), \sigma(m))$ is free of prime factors congruent to $3 \mod 4$, we can write
\begin{align*}
\prod_{\substack{p \mid \alpha(m)\beta(m) \\ p \equiv 3 \mod 4}} \bigg(1 + \frac{1}{p} \bigg) &= \prod_{\substack{p \mid \alpha(m) \\ p \equiv 3 \mod 4}} \bigg(1 + \frac{1}{p} \bigg)\prod_{\substack{p \mid \beta(m) \\ p \equiv 3 \mod 4}} \bigg(1 + \frac{1}{p} \bigg) \\ &= \prod_{\substack{p \mid s(m) \\ p \equiv 3 \mod 4}} \bigg(1 + \frac{1}{p} \bigg)\prod_{\substack{p \mid \sigma(m) \\ p \equiv 3 \mod 4}} \bigg(1 + \frac{1}{p} \bigg).
\end{align*}
Now, since $\sigma(m)$ is divisible by all primes $p \leq \sqrt{\log_2 x}$ by property (i) from Proposition \ref{m conditions}, it follows from Corollary \ref{mult mertens in progressions} that
\[
\prod_{\substack{p \mid \sigma(m) \\ p \equiv 3 \mod 4}} \bigg(1 + \frac{1}{p} \bigg) \gg \sqrt{\log_3 x}.
\]
By simply ignoring the product over $p \mid s(m)$, we obtain a lower bound of the form
\begin{align}\label{lower bound m sum}
\frac{x(\log_3 x)^{1/2}}{(\log x)^{3/2}} \sideset{}{'} \sum_{m \leq x^{1/500}} \frac{1}{m}.
\end{align}


For the sum, note that
\[
\sideset{}{'} \sum_{m \leq x^{1/500}} \frac{1}{m} \gg \sideset{}{^\flat} \sum_{m \leq x^{1/500}} \frac{1}{m} - \sideset{}{^\sharp} \sum_{m \leq x^{1/500}} \frac{1}{m},
\]
where the $\flat$ symbol indicates that the sum is over those $m = m_1m_2$ satisfying \emph{only} the decomposition conditions in the bulleted list above, and the $\sharp$ symbol indicates a sum over those $m$ (satisfying the decomposition conditions) for which:
\begin{itemize}
\item[(a)] there exists a prime $q \mid (m, \sigma(m))$, $q \equiv 3 \mod 4$;
\item[(b)] $m$ fails condition (i) from Proposition \ref{m conditions}; or
\item[(c)] $m$ fails to satisfy the bound \eqref{m isnt small}.
\end{itemize}

We bound the $\sharp$ sum from above, to show we are not subtracting too much. We deal first with condition (a). If a prime $q \equiv 3 \pmod 4$ divides $(m, \sigma(m))$, then $q \mid m_2$, and hence $m \geq q > (\log_2 x)^{10}$. For $t \geq (\log_2 x)^{10}$, let $A(t)$ denote the number of $m \leq t$ such that there exists such a prime $q \mid (m, \sigma(m))$. From the proof of Lemma 2.2 of \cite{tro15}, $A(t) = O(t/(\log_2 x)^4)$. An application of partial summation and integration by parts yields
\begin{align*}
\sum_{\substack{(\log_2 x)^{10} < m \leq x^{1/500} \\ \exists \, q > (\log_2 x)^{10} \, , \, q \mid (m, \sigma(m))}} \frac{1}{m} \ll \int_{(\log_2 x)^{10}}^{x^{1/500}} \frac{A(t)}{t^2} \, dt \ll \frac{1}{(\log_2 x)^4} \int_{(\log_2 x)^{10}}^{x^{1/500}} \frac{1}{t} \, dt \ll \frac{\log x}{(\log_2 x)^4}.
\end{align*}

Now suppose $m$ falls into group (b), but not (c). As before, we set $L = \lceil \log x \rceil$ and consider $\lambda \in \cL = \{1/L, \ldots, (L - 1)/L\}$. The number of $m$ in the interval $(x^\lambda, ex^\lambda]$ for which one of the properties (i) -- (iii) fails is $O(x^\lambda/(\log_2 x)^4)$ (cf. equation \eqref{m in lambda failing}), and hence, summing over $m$ satisfying (b) but not (c) yields (with $y = \exp(\log x / (\log_2 x)^4)$)
\[
\sum_{\substack{y < m \leq x^{1/500} \\ m \text{ satisfies (b)}}} \frac{1}{m} \leq \sum_{\lambda \in \cL} \sum_{\substack{x^\lambda < m \leq ex^\lambda \\ m \text{ satisfies (b)}}} \frac{1}{m} \ll \sum_{\lambda \in \cL} \frac{1}{x^\lambda} \cdot \frac{x^\lambda}{(\log_2 x)^4} \ll \frac{\log x}{(\log_2 x)^4}.
\]

Finally, if $m$ is counted by (c), we have the easy bound
\[
\sum_{m \leq \exp(\log x / (\log_2 x)^4)} \frac{1}{m} \ll \frac{\log x}{(\log_2 x)^4}.
\]

Thus, we have shown that
\[
\sideset{}{^\sharp} \sum_{m \leq x^{1/500}} \frac{1}{m} \ll \frac{\log x}{(\log_2 x)^4}.
\]
This upper bound, combined with the factor of $\frac{x(\log_3 x)^{1/2}}{(\log x)^{3/2}}$ from \eqref{lower bound m sum}, results in a term of the shape
\[
\frac{x(\log_3 x)^{1/2}}{(\log x)^{3/2}}\sideset{}{^\sharp} \sum_{m \leq x^{1/500}} \frac{1}{m} \ll \frac{x(\log_3 x)^{1/2}}{(\log x)^{3/2}} \cdot \frac{\log x}{\log_3 x} = o\Big( \frac{x}{(\log x)^{1/2}} \Big),
\]
which is negligible.

It remains to bound the quantity
\[
\frac{x(\log_3 x)^{1/2}}{(\log x)^{3/2}} \sideset{}{^\flat} \sum_{m \leq x^{1/500}} \frac{1}{m}
\]
from below. It is clear that
\[
\sideset{}{^\flat} \sum_{m \leq x^{1/500}} \frac{1}{m} \gg \bigg( \sum_{\substack{m_1 \leq x^{1/1000} \\ m_1 \sqrt{\log_2 x}-\text{smooth} \\ p \mid m_1 \implies p \equiv 1 \mod 4 \\ m_1 \text{ squarefree}}} \frac{1}{m_1} \bigg) \bigg( \sum_{\substack{m_2 \leq x^{1/1000} \\ m_2 \equiv 3 \mod 4 \\ p \mid m_2 \implies p > (\log_2 x)^{10}}} \frac{1}{m_2} \bigg).
\]
We first consider the sum over $m_1$. Notice that
\[
\prod_{\substack{p \leq \sqrt{\log_2 x} \\ p \equiv 1 \mod 4}} p = \exp\bigg( \sum_{\substack{p \leq \sqrt{\log_2 x} \\ p \equiv 1 \mod 4}} \log p \bigg) \leq \exp(c \sqrt{\log_2 x}) < x^{1/1000}
\]
for some absolute constant $c > 0$ and sufficiently large $x$. Therefore, the sum over $m_1$ is
\[
\sum_{\substack{m_1 \leq x^{1/1000} \\ m_1 \sqrt{\log_2 x}-\text{smooth} \\ p \mid m_1 \implies p \equiv 1 \mod 4 \\ m_1 \text{ squarefree}}} \frac{1}{m_1} = \prod_{\substack{p \leq \sqrt{\log_2 x} \\ p \equiv 1 \mod 4}} \bigg(1 + \frac{1}{p} \bigg) \asymp (\log_3 x)^{1/2},
\]
again by Corollary \ref{mult mertens in progressions}. For the sum over $m_2$, a quick application of Brun's sieve shows that
\[
\#\{ m_2 \leq z : p \mid m_2 \implies p \in [(\log_2 x)^{10}, z], m_2 \equiv 3 \mod 4 \} \gg \frac{z}{\log_3 x}
\]
for all $z \geq \log x$, say. Then, by partial summation,
\[
\sideset{}{'} \sum_{m_2 \leq x^{1/1000}} \frac{1}{m_2} \gg \frac{\log x}{\log_3 x}.
\]
Combining all of our estimates, we obtain 
\[
\frac{x(\log_3 x)^{1/2}}{(\log x)^{3/2}} \sideset{}{^\flat} \sum_{m \leq x^{1/500}} \frac{1}{m} \gg \frac{x(\log_3 x)^{1/2}}{(\log x)^{3/2}} \cdot (\log_3 x)^{1/2} \cdot \frac{\log x}{\log_3 x} = \frac{x}{\sqrt{\log x}},
\]
as desired.

\bigskip

\noindent\textbf{Acknowledgements.} The author wishes to thank Paul Pollack for suggesting the problem addressed in this paper and for many helpful conversations during the preparation of this manuscript.

\bibliographystyle{amsplain}
\bibliography{refs}

\end{document}